\documentclass[preprint,10pt]{elsarticle}
\usepackage{times}

\usepackage{cite}
\usepackage{amsmath}
\usepackage{array}

\usepackage{lineno,hyperref}
\usepackage{amsmath}
\usepackage{amssymb}
 \usepackage{amsthm}
\usepackage{mfl,color}
 \usepackage{soul}

 \usepackage{mathrsfs}

\usepackage{color}

\usepackage{tikz}
\usetikzlibrary{arrows,chains,matrix,positioning,scopes}
\def\BibTeX{{\rm B\kern-.05em{\sc i\kern-.025em b}\kern-.08em
    T\kern-.1667em\lower.7ex\hbox{E}\kern-.125emX}}

\makeatletter
\tikzset{join/.code=\tikzset{after node path={%
\ifx\tikzchainprevious\pgfutil@empty\else(\tikzchainprevious)%
edge[every join]#1(\tikzchaincurrent)\fi}}}
\makeatother

\tikzset{>=stealth',every on chain/.append style={join},
         every join/.style={->}}
\tikzstyle{labeled}=[execute at begin node=$\scriptstyle,
   execute at end node=$]

\newtheorem{Thm}{Theorem}
\newtheorem{Rmk}[Thm]{Remark}
\newtheorem{Exm}[Thm]{Example}
\newtheorem{Cor}[Thm]{Corollary}
\newtheorem{Lem}[Thm]{Lemma}
\newtheorem{Pro}[Thm]{Proposition}
\newtheorem{Def}[Thm]{Definition}

\begin{document}

\begin{frontmatter}
\title{New foundations of reasoning \\ via real-valued first-order logics}

\author{Guillermo Badia}

\address{
 University of Queensland\\ 
 Brisbane, St Lucia, QLD 4072, Australia\\ 
\texttt{guillebadia89@gmail.com} }

\author{Ronald Fagin}

\address{ IBM Almaden Research Center \\ \texttt{fagin@us.ibm.com}}

\author{Carles Noguera}

\address{University of Siena\\
          San Niccol\`o, via Roma 56\\53100 Siena, Italy\\
           \texttt{carles.noguera@unisi.it}}

\begin{abstract}
Many-valued logics in general, and real-valued logics in particular, usually focus on a notion of consequence based on preservation of full truth, typically represented by the value $1$ in the semantics given in the real unit interval $[0,1]$. In a recent paper (\emph{Foundations of Reasoning with Uncertainty via Real-valued Logics},
arXiv:2008.02429v2, 2021), Ronald Fagin, Ryan Riegel, and Alexander Gray have introduced a new paradigm that allows to deal with inferences in {\em propositional} real-valued logics based on a rich class of sentences, multi-dimensional sentences, that talk about combinations of any possible truth-values of real-valued formulas. They have given a sound and complete axiomatization that tells exactly when a collection of combinations of  truth-values of formulas imply another combination of truth-values of formulas. In this paper, we extend their work to the first-order (as well as modal) logic of multi-dimensional sentences. We give axiomatic systems and prove corresponding completeness theorems, first assuming that the structures are defined over a fixed domain, and later for the logics of varying domains. As a by-product, we also obtain a 0-1 law for finitely-valued versions of these logics.

\end{abstract}


\end{frontmatter}




\section{Introduction}
Typically the study of inference in many-valued logic answers the following question: given that all premises in a given set $\Gamma$ are {\em fully} true (i.e.\ take value $1$), what other formulas $\gamma$ can we see to be fully true as a consequence? This standard approach can be deemed unsatisfying because, when it comes to valid inference, it disregards almost all of the rich structure of truth-values and concentrates only on preservation of the value $1$. A natural question would be instead: what information can we infer on the assumption that the formulas in $\Gamma$ are {\em partially} true (i.e.\ take truth-values other than $1$)? What other formulas could be seen to be partially true or completely false?

In fact, the work in~\citep{fagin} poses the above question not just for single formulas but for sequences of formulas taking any combinations of truth-values considered as a single expression called a \emph{multi-dimensional sentence} (in short, an MD-sentence). More precisely, an MD-sentence is a syntactic object of the form $\langle \sigma_1, \dots, \sigma_k; S\rangle$ where $S$ (called the {\em information set}) is a set of $k$-tuples of truth-values for the sequence of formulas $\sigma_1,\dots,\sigma_k$ (called the {\em components}). The semantic intuition is that $\langle \sigma_1, \dots, \sigma_k; S\rangle$ should be true in an interpretation if the sequence of truth-values that  $\sigma_1, \dots, \sigma_k$ take in that interpretation is one of the $k$-tuples in $S$. 
As it happens, MD-sentences of the form 
$\langle \sigma; S\rangle$ where $S$ is a union of a finite number of closed intervals (and where the connectives have been given the \L ukasiewicz semantics) can be expressed in the extension of \L ukasiewicz logic known as Rational Pavelka logic. However, if, for example, $S$ is a left open interval, $(0.5,1]$, then it was shown in~\citep{fagin} that Rational Pavelka logic is unable to deal with this.
\footnote{Pavelka introduced in~\citep{Pavelka} a formal system of real-valued logic that later was found to be equivalent to the expansion of \L ukasiewicz logic by enriching the language with a truth-constant $\overline{r}$ for each real number in $r \in [0,1]$ with certain additional axioms and proved a corresponding completeness theorem. Nov\'ak extended Pavelka's logic to a first-order language thus obtaining a corresponding expansion of \L ukasiewicz first-order logic~\citep{Novak:CompletenessFirstOrder}. The enriched language of these systems allows to write sentences of the form $\overline{r} \to \varphi$ and $\varphi \to \overline{r}$ and hence allows to stipulate in a syntactical manner the precise truth-value $r$ that the formula $\varphi$ has to take in a model, but of course in a much more restricted way that in MD-sentences, as just discussed above. Finally, it is worth mentioning that H\'ajek greatly simplified Pavelka's and Nov\'ak's approach in~\citep{haj} by showing that essentially the same results could be obtained in a countable language that included truth-constants only for the rational numbers in $[0,1]$.} 

The goal of~\citep{fagin} was to axiomatize inference genuinely involving many truth-values. The authors indeed provided an axiomatization in terms of MD-sentences in a parametrized way that captures essentially all real-valued logics. In the present article, we generalize the work in~\citep{fagin} to the first-order and modal contexts. 

The article is arranged as follows. First, in \S~\ref{s:Propositional}, we give a fast overview of the necessary notions and results that we borrow from the propositional case studied in~\citep{fagin}. In \S~\ref{sec1} we study the first-order (as well as modal) logic of multi-dimensional sentences (generalizing  the definition of~\citep{fagin}) when the models considered all have the same fixed domain (which may be of any fixed cardinality, either finite or infinite). The key result is a completeness result that follows the strategy of that in~\citep{fagin} for the propositional case.  In \S~\ref{sec2}  we show how our approach leads to parameterized axiomatizations of the valid finitary inferences of many prominent first-order real-valued logics. Since this includes several logics that are not recursively enumerable for validity, our system in general does not yield a recursive enumeration of theorems. In \S \ref{sec3}, we prove a 0-1 law for finitely-valued versions of the logics dealt with in \S~\ref{sec1}. Finally, in \S~\ref{sec4} we remove the restriction of a fixed domain and provide a completeness theorem for the first-order  logic of multi-dimensional sentences on arbitrary domains.

\section{The propositional case: an overview}\label{s:Propositional}

This section presents a brief summary of the key results and notions from~\citep{fagin}. Following that article, we take a (propositional) \emph{multi-dimensional sentence} (in symbols, an MD-sentence) to be an expression  of the form $\langle \sigma_1, \dots, \sigma_k; S\rangle$ where $S \subseteq [0,1]^k$. For a fixed $k$, we may speak of $k$-dimensional sentences.

The semantics of MD-sentences is as follows. By a {\em model} we mean an assignment  $\model{M}$ from atomic sentences of a propositional language $\mathcal{L}$ to truth-values from $[0,1]$. The usual real-valued logics (\L ukasiewicz, Product, G\"odel, etc.) all have inductive definitions indicating how to assign values to all formulas and hence the notion of the value of an arbitrary formula in the language $\mathcal{L}$ in a given model $\model{M}$ is well-defined. Next, for an  MD-sentence $\langle \sigma_1, \dots, \sigma_k; S\rangle$, we say that $\model{M}$ \emph{satisfies} this sentence (in symbols, $\model{M}\models \langle \sigma_1, \dots, \sigma_k; S\rangle$) if $\tuple{s_1, \dots s_k} \in S$ where $s_i$ ($1 \leq i \leq k$) is the value in $\model{M}$ of $\sigma_i$.

Given these definitions one can consider Boolean combinations of MD-sentences. For example, take $\gamma_1 := \langle \sigma^1_1, \dots, \sigma^1_n; S_1\rangle$ and $\gamma_2 := \langle \sigma^2_1, \dots, \sigma^2_m; S_2\rangle$. Then, we may say that $\model{M}\models \gamma_1 \wedge \gamma_2$ iff  $\model{M}\models \gamma_1$ and $\model{M}\models \gamma_2$. An interesting result from~\citep{fagin} is that MD-sentences are closed under Boolean combinations, in the sense that for any Boolean combination of such sentences there is an MD-sentence equivalent to such combination. Hence, the collection of MD-sentences is expressively quite robust.

\begin{Exm}
An easy example of a valid MD-sentence in, say, G\"odel semantics, is the 3-dimensional sentence $\langle A, B, A\vee B; S\rangle$ where $S$ is the set of all triples $\tuple{s_1, s_2, s_3}$ where $s_1, s_2 \in [0,1]$ and $s_3$ is the maximum of the set $\{s_1, s_2\}$.
\end{Exm}

Now it is natural to try to build a calculus that will capture exactly the valid finitary inferences involving MD-sentences. This is what we do next.

\paragraph{Axioms} We have only one axiom schema:
\begin{itemize}
\item[(1)] $\tuple{\sigma_1, \ldots, \sigma_k, [0,1]^k }$.
\end{itemize}
\paragraph{Inference rules} 

\begin{itemize}
\item[(2)] From $\tuple{\sigma_1, \ldots, \sigma_k; S}$ infer $\tuple{\sigma_{\pi (1)}, \ldots, \sigma_{\pi (k)}; S'}$,
\end{itemize}
where $S' = \{\tuple{s_{\pi (1)}, \ldots, s_{\pi (k)}} \mid \tuple{s_{1}, \ldots, s_{k}} \in S\}$ and $\pi$ is a permutation of $1, \ldots, k$.

\begin{itemize}
\item[(3)] From  $\tuple{\sigma_1, \ldots, \sigma_k; S}$ infer $$\langle\sigma_1, \ldots, \sigma_k, \sigma_{k+1}, \ldots, \sigma_m; S \times  [0,1]^{m-k}\rangle.$$
\end{itemize}

\begin{itemize}
\item[(4)] From $\tuple{\sigma_1, \ldots, \sigma_k; S_1}$ and  $\tuple{\sigma_1, \ldots, \sigma_k; S_2}$ infer $\tuple{\sigma_1, \ldots, \sigma_k; S_1 \cap S_2}.$
\end{itemize}
\begin{itemize}
\item[(5)] For  $0< r < k$, from $\tuple{\sigma_1, \ldots, \sigma_k; S}$  infer $\tuple{\sigma_1, \ldots, \sigma_{k-r}; S'}$, where $S'= \{\tuple{s_1, \ldots, s_{k-r}} \mid \tuple{s_1, \ldots, s_{k}} \in S\}$.
\end{itemize}
\begin{itemize}
\item[(6)] From $\tuple{\sigma_1, \ldots, \sigma_k; S}$ infer $\tuple{\sigma_1, \ldots, \sigma_k; S'}$, when $S \subseteq S'$.
\end{itemize}

Finally, before we introduce the last rule, let us define a piece of  notation.  For any $j$-ary connective $\circ$,  from a real-valued logic and real numbers $s_1, \ldots, s_j$ from $[0,1]$ we can define the function $\circ (s_1, \ldots, s_j)$  giving as output what the connective $\circ$  indicates in a given real-valued logic for the values $s_1, \ldots, s_j$. 

Before we introduce the last rule, we need an auxiliary notion. Given an MD-sentence $\tuple{\sigma_1, \ldots, \sigma_k; S}$, we say that a tuple $\tuple{s_1, \ldots, s_k} \in S$ is {\em good} if $s_m = \circ (s_{m_1}, \ldots,   s_{m_j})$  whenever $\sigma_m = \circ(\sigma_{m_1}, \ldots, \sigma_{m_j})$ (for any $m$-ary connective $\circ$ and for any $m$). Notice that this is a local property of each tuple in $S$, in the sense that it does not depend on what other tuples are in the information set. Now, the last inference rule is
\begin{itemize}
\item[(7)] From $\tuple{\sigma_1, \ldots, \sigma_k; S}$  infer $\tuple{\sigma_1, \ldots, \sigma_k; S'}$, where $S'$ is the set of good tuples in $S$.  
\end{itemize}

A \emph{proof} of an MD-sentence $\gamma$ from a set $\Gamma$ of MD-sentences in this system  consists, as usual, of a finite sequence of MD-sentences such that the last member is $\gamma$ and every element of the sequence is either an axiom, one of the member of $\Gamma$, or it follows from previous elements by one of the inference rules. We write $\Gamma \vdash \gamma$ to indicate that there exists a proof of $\gamma$ from $\Gamma$. 

The central result from~\citep{fagin} states that if $\Gamma$ is a \emph{finite} set of MD-sentences, we have that $\Gamma \vdash \gamma$ is equivalent to $\Gamma \vDash \gamma$. It is noteworthy that this technique provides a  parameterized way of building calculi for MD-sentences with semantics for the standard real-valued logics  (where the parameters give a particular semantic meaning to the connectives of the language); special extra steps need to be taken for the logic of probabilities, as discussed in~\citep{fagin}. The restriction to finite sets is necessary due to the finitary character of \L ukasiewicz logic. Finally, in~\citep{fagin} a decision procedure for validity in this system of MD-sentences for G\"odel and \L ukasiewicz semantics is introduced. Furthermore, the  algorithm of the procedure is implemented and tested on various interesting cases.

\section{The logic of a fixed domain}\label{sec1}
Throughout this section, let $M$ be  any  fixed set, \emph{finite} or \emph{infinite}. Observe that for finite fixed domains, by means of eliminating quantifiers (turning a universal quantifier into a big conjunction and turning an existential quantifier into a big disjunction), we could use an approach that essentially reduces the problem to what was done in~\citep{fagin}. We work with a first-order relational vocabulary  $\tau$ to simplify things (but everything we do can be adjusted to accommodate function and constant symbols).

\subsection{First-order case (the logic of a fixed domain)}
This part is devoted to provide an axiomatization of  the logic of a fixed domain $M$ (of any cardinality), in the sense of the valid inferences over all models with domain $M$. We will construct the set $\mathrm{MD}(M)$ of \emph{MD-sentences} with domain $M$.
Let $\mathrm{MD}(M)$ contain all sentences of the form $\tuple{\f_1(\overline{x}_{\f_1}), \ldots, \f_k(\overline{x}_{\f_k}); S}$
where $\overline{x}_{\f_i}:= x_{i_{1}}, \ldots, x_{i_{n_i}}$, and $S \subseteq [0,1]^{M^{n_1}} \times \ldots \times [0,1]^{M^{n_k}}$\!. In the expression $\f_i(\overline{x})$, the free variables of $\f_i$ (if any) will be exactly those in the list $\overline{x}_{\f_i}$.
When $\overline{x}_{\f_i}$  is empty, $\f_i$ is a sentence and what it gets assigned in a given $S$ is simply a nullary function, in other words, an element of $[0,1]$, as in the propositional case. If none of the formulas $\f_i$ in the MD-sentence $\tuple{\f_1, \ldots, \f_k; S}$ contains free variables, then the situation is exactly as in the propositional case~\citep{fagin} and there is no need to mention in $S$  the set $M$.

\begin{Exm}\label{Ex1} Take a vocabulary $\tau$ with only two unary predicates $P$ and $U$. Then, we can build the sentence $\tuple{Px, \forall x\, Ux; S}$ where $S =\{\tuple{f, r} \mid \text{$r \in [0.5, 0.8)$}$, $f$ is a mapping with domain $M$ and range included in the set $[0,1]$\}. 
Intuitively, we want this sentence to be satisfied in a model $\model{M}$ with domain $M$ if the truth-value of\/ $\forall x\, Ux$ is a real number in the interval $[0.5, 0.8)$ and the interpretation of the predicate $P$ is a mapping from $M$ into $[0,1]$.
\end{Exm}
\begin{Def}\label{model}
\emph{A real-valued first-order \emph{model}  $\model{M}$ is a structure with domain $M$ and interpretations for the predicates of our vocabulary $\tau$ being mappings from Cartesian products of $M$ into $[0,1]$. More precisely, for an $n$-ary predicate $R$, its interpretation in $\model{M}$ is a mapping $R_\model{M}\colon M^n \longrightarrow [0,1]$.} 
\end{Def}
Inductively, using the semantics of the real-valued logic in question, one can define the truth-value of any formula for a sequence of elements $\overline{a}$ from $M$ and write it as $\semvalue{\f[\overline{a}]}_\model{M}$:

\vspace{0.5ex} 

\begin{itemize}
%

\item $\|P[\overline{a}]\|_{\model{M}}=P_{\model{M}}(\overline{a})$, for each $P\in Pred_{\tau}$;

\item $\|\circ( \f_1, \ldots, \f_n)[\overline{a}]\|_{\model{M}}= \circ(\|\f_1[\overline{a}]\|_{\model{M}}, \ldots, \|\f_n[\overline{a}]\|_{\model{M}})$, for $n$-ary connective $\circ$;


\item $\semvalue{(\forall x)\f [\overline{a}]}_{\model{M}}=\inf\{\|\f [\overline{a}, e]\|_{\model{M}}\mid e\in M\}$;

\item $\|(\exists x)\f [\overline{a}]\|_{\model{M}}=\sup\{\|\f [\overline{a}, e]\|_{\model{M}}\mid e\in M\}$.
\end{itemize}

\vspace{0.5ex} 

Whenever the vocabulary includes the equality symbol $\approx$, its semantics is defined in the following way:

\vspace{0.5ex} 

\begin{itemize}
\item $\| (x\approx y) [d,e] \|_{\model{M}}= 1$ iff $d =e$, for any $d,e \in M$.
\item $\| (x\approx y) [d,e] \|_{\model{M}}= 0$ iff $d \neq e$, for any $d,e \in M$.
\end{itemize}

\vspace{0.5ex} 
The definition of the truth-value of a quantified formula as the infimum or the supremum of the truth-values of its instances is customary in many-valued logics as a natural generalization of the semantics of quantifiers in classical logic.

A formula $\f(x_1, \ldots, x_n)$ can be said to be \emph{interpreted} in the model $\model{M}$ by the mapping $f_\f\colon M^n \longrightarrow [0, 1]$ defined as $\tuple{a_1, \ldots, a_n} \mapsto \semvalue{\f[a_1, \ldots, a_n]}_\model{M}$ (we also say that $\f(x_1, \ldots, x_n)$ \emph{defines} the mapping $f_\f$ in the model $\model{M}$).
 
Next,  take a sentence $\tuple{\f_1(\overline{x}_{\f_1}), \ldots, \f_k(\overline{x}_{\f_k}); S}$.
Then, we may write  $$\model{M} \models \tuple{\f_1(\overline{x}_{\f_1}), \ldots, \f_k(\overline{x}_{\f_k}); S}$$ if the formulas $\f_1(\overline{x}_{\f_1}), \ldots, \f_k(\overline{x}_{\f_k})$ respectively define mappings $f_1, \ldots, f_k$ in the model $\model{M}$ and $\tuple{f_1, \ldots, f_k} \in S$. Notice that, if any of the $\f_i$s is a sentence, then the corresponding $f_i$ is a constant function.
If all the $\f_i$s are sentences, this definition basically boils down to what appears in~\citep{fagin}.
 
We introduce now a proof system associated to the domain $M$, called the {\em MD-system of $M$}, by considering the axioms and inference rules given in~\citep{fagin} for the propositional case and  modifying only what is needed:

\paragraph{Axioms} We have only one axiom schema:
\begin{itemize}
\item[(1)] $\tuple{\f_1(\overline{x}_{\f_1}), \ldots, \f_k(\overline{x}_{\f_k}), [0,1]^{M^{n_1}} \times \ldots \times [0,1]^{M^{n_k}}}$ for all formulas\\ $\f_1(\overline{x}_{\f_1}), \ldots, \f_k(\overline{x}_{\f_k})$.
\end{itemize}
\paragraph{Inference rules} 

\begin{itemize}
\item[(2)] From $$\tuple{\f_1(\overline{x}_{\f_1}), \ldots, \f_k(\overline{x}_{\f_k}); S}$$ infer $$\tuple{\f_{\pi (1)}(\overline{x}_{\f_{\pi (1)}}), \ldots, \f_{\pi (k)}(\overline{x}_{\f_{\pi (k)}}); S'},$$ where $S' = \{\tuple{f_{\pi (1)}, \ldots, f_{\pi (k)}} \mid \tuple{f_{1}, \ldots, f_{k}} \in S\}$ and $\pi$ is a permutation of $1, \ldots, k$.
\end{itemize}

\begin{itemize}
\item[(3)] From $$\tuple{\f_1(\overline{x}_{\f_1}), \ldots, \f_k(\overline{x}_{\f_k}); S}$$ infer $$\langle\f_1(\overline{x}_{\f_1}), \ldots, \f_k(\overline{x}_{\f_k}), \f_{k+1}(\overline{x}_{\f_{k+1}}), \ldots, \f_m(\overline{x}_{\f_m}); S \times [0,1]^{M^{n_{k+1}}} \times \ldots \times [0,1]^{M^{n_{m}}}\rangle.$$
\end{itemize}

\begin{itemize}
\item[(4)] From $$\tuple{\f_1(\overline{x}_{\f_1}), \ldots, \f_k(\overline{x}_{\f_k}); S_1}$$ and $$\tuple{\f_1(\overline{x}_{\f_1}), \ldots, \f_k(\overline{x}_{\f_k}); S_2}$$ infer $$\tuple{\f_1(\overline{x}_{\f_1}), \ldots, \f_k(\overline{x}_{\f_k}); S_1 \cap S_2}.$$
\end{itemize}
\begin{itemize}
\item[(5)] For  $0< r < k$, from $$\tuple{\f_1(\overline{x}_{\f_1}), \ldots, \f_k(\overline{x}_{\f_k}); S}$$  infer $$\tuple{\f_1(\overline{x}_{\f_1}), \ldots, \f_{k-r}(\overline{x}_{\f_{k-r}}); S'},$$ where $S'= \{\tuple{f_1, \ldots, f_{k-r}} \mid \tuple{f_1, \ldots, f_{k}} \in S\}$.
\end{itemize}
\begin{itemize}
\item[(6)] From $$\tuple{\f_1(\overline{x}_{\f_1}), \ldots, \f_k(\overline{x}_{\f_k}); S}$$  infer $$\tuple{\f_1(\overline{x}_{\f_1}), \ldots, \f_k(\overline{x}_{\f_k}); S'}$$ where $S \subseteq S'$.
\end{itemize}

Finally, before we introduce the last rule, let us define a piece of  notation. Consider an arbitrary domain $M$ and functions $f_1, \ldots, f_j$  from some Cartesian products of $M$ into $[0,1]$. Then, for any $j$-ary connective $\circ$ from a real-valued logic, we can define the function $\circ (f_1, \ldots, f_j)$ as taking arguments componentwise as indicated by the output of the $f_i$s ($i \in \{1, \ldots, j\}$) and giving as output what $\circ$  indicates. 
Also, we need to generalize also the notion of good tuple. Indeed, given an MD-sentence $\tuple{\f_1(\overline{x}_{\f_1}), \ldots, \f_k(\overline{x}_{\f_k}); S}$, we say that a tuple $\tuple{f_1, \ldots, f_{k}} \in S$ is {\em good} if
\begin{itemize}
\item[(a)] $f_m = \circ (f_{m_1}, \ldots,   f_{m_j})$ whenever $\f_m(\overline{x}_{\f_m}) = \circ(\f_{m_1}(\overline{x}_{\f_{m_1}}), \ldots, \f_{m_j}(\overline{x}_{\f_{m_j}})),$
\item[(b)]   $f_i(e_{1}, \ldots, e_{n_j})= \inf\{f_j(e_{1}, \ldots, e_{n_j}, e) \mid e \in M\}$ whenever\\ $\f_i(\overline{x}_{\f_i}) = \forall y\, \f_j(\overline{x}_{\f_j}),$ for all $e_{1}, \ldots, e_{n_j} \in M^{n_j}$, 
\item[(c)]   $f_i(e_{1}, \ldots, e_{n_j})= \sup\{f_j(e_{1}, \ldots, e_{n_j}, e) \mid e \in M\}$   whenever\\ $\f_i(\overline{x}_{\f_i}) = \exists y\, \f_j(\overline{x}_{\f_j})$, for all $e_{1}, \ldots, e_{n_j} \in M^{n_j}$.
\end{itemize}

\begin{itemize}
\item[(7)] From $$\tuple{\f_1(\overline{x}_{\f_1}), \ldots, \f_k(\overline{x}_{\f_k}); S}$$  infer $$\tuple{\f_1(\overline{x}_{\f_1}), \ldots, \f_k(\overline{x}_{\f_k}); S'},$$ where $S'$ is the set of good tuples in $S$. 
\end{itemize}

A \emph{proof} of an MD-sentence $\gamma$ from a set $\Gamma$ of MD-sentences in this system  consists, as usual, of a finite sequence of MD-sentences such that the last member is $\gamma$ and every element of the sequence is either an axiom, one of the member of $\Gamma$, or it follows from previous elements by one of the inference rules. We write $\Gamma \vdash_M \gamma$ to indicate that there exists a proof of $\gamma$ from $\Gamma$. 


\begin{Lem}\label{min} Let 
$\tuple{\f_1(\overline{x}_{\f_1}), \ldots, \f_k(\overline{x}_{\f_k}); S}$ be the premise of Rule~(7) and assume that $G=\{\f_1(\overline{x}_{\f_1}), \ldots, \f_k(\overline{x}_{\f_k})\}$ is closed under subformulas in the usual sense. Then, the conclusion $\tuple{\f_1(\overline{x}_{\f_1}), \ldots, \f_k(\overline{x}_{\f_k}); S'}$ is \emph{minimized}, i.e., whenever  $\tuple{f_1, \ldots, f_k}\in S'$, there is a model $\model{M}$ (with domain $M$) of  $\tuple{\f_1(\overline{x}_{\f_1}), \ldots, \f_k(\overline{x}_{\f_k}); S'}$  such that for $1\leq i \leq k$ the interpretation of $\f_i(\overline{x}_{\f_i})$ is $f_i$.
\end{Lem}

\begin{proof} Assume that $\tuple{f_1, \ldots, f_k}\in S'$. Since $G$ is closed under subformulas, there is a subsequence of $\tuple{f_1, \ldots, f_k}$ that determines interpretations on the domain $M$ for the atomic formulas appearing in $G$, i.e., interpretations for the predicates of $\tau$. But this subsequence then defines a model $\model{M}$ based on the domain $M$ where the interpretations of $\f_1(\overline{x}_{\f_1}), \ldots, \f_k(\overline{x}_{\f_k})$ are as indicated by $\tuple{f_1, \ldots, f_k}$. This is because Rule~(7) is designed to select only those sequences  $\tuple{f_1, \ldots, f_k}$ that respect the semantics of the underlying real-valued logic.
\end{proof}

\begin{Rmk}\emph{
Lemma~\ref{min} plays an important role in the completeness argument in this general framework. Roughly speaking, it relies on the fact that the set $S'$ can encode a  model for a series of formulas $\f_1(\overline{x}_{\f_1}), \ldots, \f_k(\overline{x}_{\f_k})$ with domain $M$ by a sequence of interpretations to the finite list of predicates appearing in such formulas in a way that is consistent with the semantics of the underlying real-valued logic.  It is not difficult to see that, for a finite vocabulary $\tau$, we can find a set $S$ encoding all possible models with domain $M$. For example if $\tau$ is the set $\{P_1, \dots, P_k\}$ of predicates, we can take $S$ to be the set of all sequences $\tuple{f_1, \ldots, f_k}$ of possible interpretations of the predicates from our list on the domain $M$.}
\end{Rmk}

Similarly to~\citep[Lemma 5.3]{fagin} we obtain:

\begin{Lem}\label{equiv}The conclusion and premises of rules (2), (3), (4), and (7) are logically equivalent.

\end{Lem}

\begin{proof} We will sketch the argument for Rule~(7) and leave the rest as exercises to the reader. Let $\model{M}$ be a model such that
$$\model{M} \models \tuple{\f_1(\overline{x}_{\f_1}), \ldots, \f_k(\overline{x}_{\f_k}); S}.$$ But then the interpretations $f_1, \ldots, f_{k}$ of the formulas  $\f_1(\overline{x}_{\f_1}), \ldots, \f_k(\overline{x}_{\f_k})$ in the model $\model{M}$ respect the semantics of the connectives and quantifiers according to the real-valued logic in question. Since $\tuple{f_1, \ldots, f_{k}}\in S$ by hypothesis, we must have  that $\tuple{f_1, \ldots, f_{k}}\in S'$ where $S'$ is as in Rule~(7). Hence, 
 $$\model{M} \models \tuple{\f_1(\overline{x}_{\f_1}), \ldots, \f_k(\overline{x}_{\f_k}); S'},$$
 as desired. On the other hand, if $$\model{M} \models \tuple{\f_1(\overline{x}_{\f_1}), \ldots, \f_k(\overline{x}_{\f_k}); S'},$$ given the soundness of Rule~(6), it follows that 
 $$\model{M} \models \tuple{\f_1(\overline{x}_{\f_1}), \ldots, \f_k(\overline{x}_{\f_k}); S}.\qedhere$$
\end{proof}

The following lemma is straightforward to show.

\begin{Lem}\label{min2}
Minimization is preserved by the rules (2) and (4), i.e.\ if the premises of the rules are minimized, then their conclusions are too.
\end{Lem}

Let $\Gamma \vDash_{M} \gamma$ mean that for each model $\model{M}$ with domain $M$, if $\model{M} \models \Gamma$ then $\model{M} \models \gamma$. We call the relation $\vDash_{M}$ the {\em MD-logic of $M$}. We can now reconstruct the soundness and completeness argument from~\citep{fagin} and obtain the following theorem that the MD-system of $M$ is actually an axiomatization of the MD-logic of $M$.

\begin{Thm}[Completeness of the logic of a fixed domain]\label{t:Compl-MD-system-Countable-Domain}
Let $\Gamma$ be a {\em finite} set of MD-sentences and $\gamma$ an MD-sentence. Then, $\Gamma \vdash_M \gamma$ iff\/ $\Gamma \vDash_M \gamma$.
\end{Thm}

 \begin{proof} To see that $\Gamma \vdash_M \gamma$ only if $\Gamma \vDash_{M} \gamma$, one proceeds, as usual, by induction on the length of the proof, i.e.\ we start by showing that the axiom schema is sound and that the rules preserve the truth of the MD-sentences. For example, every instance of the axiom schema is sound since every formula in the usual first-order sense is interpreted by some mapping on a given model based on the domain $M$.
 
 To show completeness, we follow the argument on~\citep[p.~12]{fagin}.  The strategy is to transform $\Gamma$ into an equivalent MD-sentence from which $\gamma$ can be deduced. We may assume without loss of generality that $\Gamma$ is non-empty, for otherwise we could replace it by an instance of Axiom (1).

Indeed, assume that we have a finite set $\Gamma= \{\gamma_1, \ldots, \gamma_n \}$ of MD-sentences in which, for each $i \in \{1, \ldots, n\}$, $\gamma_i$ is  the MD-sentence $\tuple{\f^i_1(\overline{x}_{\f_1}), \ldots, \f^i_k(\overline{x}_{\f_{k_i}}); S_i}$. Suppose further that  $\gamma$ is $\tuple{\f^0_1(\overline{x}_{\f_1}), \ldots, \f^0_k(\overline{x}_{\f_{k_0}}); S_0}$. Then, take the sets $\Gamma_i=\{\f^i_1(\overline{x}_{\f_1}), \ldots, \f^i_k(\overline{x}_{\f_{k_i}})\}$ and $\Gamma_0=\{\f^0_1(\overline{x}_{\f_1}), \ldots, \f^0_k(\overline{x}_{\f_{k_0}})\}$. We take $G$ to be the usual closure under subformulas of the set $\bigcup_{j\geq0}\Gamma_j$.  

$G$ is a finite set and then we can follow step by step the argument in~\citep{fagin}, applying  our slightly modified Rules~(3) and~(7). In particular, we make use of Lemma~\ref{min} instead of~\citep[Lemma 5.2]{fagin}.

For each $i$ such that  $1\leq i \leq n$, we set $H_i= G \setminus \Gamma_i$. Let $r_i$ be the cardinality of $H_i$ and suppose that $H_i =\{\theta_1(\overline{x}_{\theta_1}), \ldots, \theta_{r_i}(\overline{x}_{\theta_{r_i}})\}$.  Then, by applying Rule~(3), we can deduce the MD-sentence $$\langle\f^i_1(\overline{x}_{\f_1}), \ldots, \f^i_k(\overline{x}_{\f_{k_i}}), \f_{k+1}(\overline{x}_{\f_{k+1}}), \ldots, \f_m(\overline{x}_{\f_m}); S \times [0,1]^{M^{n_{k+1}}} \times \ldots \times [0,1]^{M^{n_{m}}}\rangle$$
from $\tuple{\f^i_1(\overline{x}_{\f_1}), \ldots, \f^i_k(\overline{x}_{\f_{k_i}}); S}$, i.e.\ $\gamma_i$, where the sequence $ \f_{k+1}(\overline{x}_{\f_{k+1}}), \ldots, \f_m(\overline{x}_{\f_m})$ is $\theta_1(\overline{x}_{\theta_1}), \ldots, \theta_{r_i}(\overline{x}_{\theta_{r_i}})$. Now let $\p_i$ be the MD-sentence that results from applying Rule~(7) to the conclusion of Rule~(3) displayed above.

Let $\f_1(\overline{x}_{\f_1}), \ldots, \f_p(\overline{x}_{\f_{p}})$ be some ordering of the formulas in  $G$; then, since  the set of first-order formulas that appear in $\p_i$ is exactly $G$, we may use Rule~(2) to turn $\p_i$ into an equivalent MD-sentence of the form $\tuple{\f_1(\overline{x}_{\f_1}), \ldots, \f_p(\overline{x}_{\f_{p}}); T_i}$, which we may denote by $\chi_i$. Furthermore, since in deriving $\chi_i$, we only appealed to rules (2), (3) and (7), by Lemma~\ref{equiv}, this MD-sentence is  logically equivalent to $\gamma_i$.

Assume that $T= T_1 \cap \ldots \cap T_n$ and define $\chi:=\tuple{\f_1(\overline{x}_{\f_1}), \ldots, \f_p(\overline{x}_{\f_{p}}); T}$. From Lemma~\ref{min}, each $\p_i$ is minimized since it comes from Rule~(7) and $$\{\f^i_1(\overline{x}_{\f_1}), \ldots, \f^i_k(\overline{x}_{\f_{k_i}}), \f_{k+1}(\overline{x}_{\f_{k+1}}), \ldots, \f_m(\overline{x}_{\f_m})\}$$ is closed under subformulas. Moreover, by Lemma~\ref{min2}, each $\chi_i$ is minimized and, hence, $\chi$ is minimized.

The MD-sentence  $\chi$  can be derived from the MD-sentences $\chi_i$ by repeated applications of Rule~(4). In fact, by Lemma~\ref{equiv}, $\chi$ and $\{\chi_1, \ldots, \chi_n\}$  have the same logical consequences, and since  $\chi_i$ is equivalent to $\gamma_i$, we have that $\{\chi_1, \ldots, \chi_n\}$  and $\{\gamma_1, \ldots, \gamma_n\} = \Gamma$ have the same logical consequences. Hence, $\chi \vDash \gamma$ given that $\Gamma \vDash \gamma$ by hypothesis. Furthermore,  in order to show that  $\Gamma \vdash \gamma$ we simply need to show that $\chi \vdash \gamma$ since $\Gamma \vdash \chi$ by the above reasoning.

 Recall that $\gamma$ is $\tuple{\f^0_1(\overline{x}_{\f_1}), \ldots, \f^0_k(\overline{x}_{\f_{k_0}}); S_0}$ and $\chi$ is  $\tuple{\f_1(\overline{x}_{\f_1}), \ldots, \f_p(\overline{x}_{\f_{p}}); T}$, so  by applying Rule~(2) we can rearrange the order of the formulas $\f_1(\overline{x}_{\f_1}), \ldots, \f_p(\overline{x}_{\f_{p}})$ so they start with $\f^0_1(\overline{x}_{\f_1}), \ldots, \f^0_k(\overline{x}_{\f_{k_0}})$ and infer from $\chi$ the MD-sentence $\chi':=\tuple{\f^0_1(\overline{x}_{\f_1}), \ldots, \f^0_k(\overline{x}_{\f_{k_0}})\ldots; T'}$. Using Lemma~\ref{equiv}, we may see that $\chi$ and $\chi'$ are logically equivalent. Hence, $\chi' \vDash \gamma$ since $\chi \vDash \gamma$. Given that $\chi$ is minimized, it follows that $\chi'$ is too  by Lemma~\ref{min2}. Using Rule~(5), from $\chi'$ we may infer an MD-sentence $\chi''$ of the form $\tuple{\f^0_1(\overline{x}_{\f_1}), \ldots, \f^0_k(\overline{x}_{\f_{k_0}}); T''}$.
 
 The final step in the proof is to show that $T'' \subseteq S_0$ for then we can use Rule~(6) to infer $\gamma$ from $\chi''$, and hence we would have $\chi \vdash \chi '\vdash \chi'' \vdash \gamma$, which means that $\chi \vdash \gamma$ as desired.
 
Assume now that $(f_1, \ldots, f_k) \in T''$ to show that $\tuple{f_1, \ldots, f_{k_0}} \in S_0$. By definition of $T''$, there is $\tuple{f_1, \ldots, f_{k_0}, \ldots, f_p} \in T'$. Given that $\chi'$ is minimized,  there is a model $\model{M}$ of $\chi'$ such that the interpretations of the formulas $\f^0_1(\overline{x}_{\f_1}), \ldots, \f^0_k(\overline{x}_{\f_{k_0}})$ are $f_1, \ldots, f_{k_0}$, respectively. Since $\chi' \vDash \gamma$, then $\model{M}\models  \gamma$, and so $\tuple{f_1, \ldots, f_{k_0}}\in S_0$.
\end{proof}

There are some subtle points to consider around what we have done, which we will discuss in the next remarks. It is important to stress that we have axiomatized the logic of {\em all models} based on the set $M$, not the logic of one particular model $\model{M}$ based on $M$. 

\begin{Rmk}\label{rule} \emph{
Let us look at the case of two-valued logic with equality (i.e.\ the classical first-order logic which, of course, is covered by our approach). Let $M$ be a finite set (say of size $n$). Now, enumerate all the first-order validities  of the form $(|M|=n) \rightarrow \f$ where $\f$ is any first-order formula and $|M|=n$ is the first-order formula saying that the size of the domain $M$ is exactly $n$. In the case of finite domains, one might modify the approach here by allowing only MD-sentences that are interval-based (in the sense of~\citep{fagin}, that is, where the sets of truth-values involved in $S$ are unions of finitely-many rational intervals) or that come from such sentences by an application of Rule~(7), making the set $\mathrm{MD}(M)$ countable, and then it is possible to show by essentially the argument in~\citep{fagin} that validity is not only recursively enumerable but decidable on such domains.}
\end{Rmk}

\begin{Rmk} \label{useful}
 \emph{Recall that satisfiability on  countably  infinite models is not recursively enumerable in two-valued first-order logic.
 Now take a first-order sentence
 $\f$ and let
  $\f_1(\overline{x}_{\f_1}), \ldots, \f_k(\overline{x}_{\f_k}), \f$ be the list of all its  subformulas. Fixing a countably infinite domain 
  $M$, we may consider now the MD-sentence $\tuple{\f_1(\overline{x}_{\f_1}), \ldots, \f_k(\overline{x}_{\f_k}), \f; S}$ (call it $\p$) where $S:= \{0,1\}^{M^{n_1}}\times \ldots \times  \{0,1\}^{M^{n_k}} \times \{1\}$. Take now the MD-sentence obtained by applying our Rule~(7) to this sentence, $\tuple{\f_1(\overline{x}_{\f_1}), \ldots, \f_k(\overline{x}_{\f_k}), \f; S'}$ (call it $\psi'$). Observe that $\p$ and $\p'$ are  equivalent. Furthermore, $\f$ has a countably infinite model iff $\p$ is satisfiable iff $\p'$ is satisfiable. Finally, by minimization and the semantics of MD-sentences, $\p'$ is satisfiable iff $S'$ is non-empty. Hence, the problem of whether an arbitrary  $S'$ is non-empty is not recursively enumerable.
 }
\end{Rmk}

Rule~(7) implies that our formal system is not finitistic in the sense of metamathematics~\citep{kleene} since when infinite domains are involved it cannot all be formalizable in arithmetic, it goes into the realm of infinitary mathematics. In this sense it is akin to an infinitary  proof system (although it does not involve infinitary formulas in the usual sense).

\subsection{Modal logic (the logic of a fixed frame)}
For this subsection, fix a frame $\mathfrak{F}:= \tuple{M, R}$ where $R \subseteq M^2$ is a binary relation on a non-empty set $M$ (finite or infinite, where we may call the elements $M$ \emph{worlds}).\footnote{In this paper we consider only this classical notion of frame, although the literature of many-valued logics has also studied natural many-valued generalizations in which $R$ would be taken as a mapping from $M^2$ to $[0,1]$ (or to other more general structures of truth-degrees); see e.g.~\citep{bou}.} Consider now a vocabulary $\tau$ consisting only of propositional variables as in modal logic and a base modal language with $\Box$ and $\Diamond$ (unlike classical logic, many-valued logics do not allow in general to define these two operators from one another). Now the set $\mathrm{MD}(M)$ of MD-sentences contains all the expressions of the form $\tuple{\f_1, \ldots, \f_k; S}$
where each $\f_i$ is a modal formula and $S \subseteq ([0,1]^{M})^k$\!.

For each real-valued model $\model{M}$-based on $\mathfrak{F}= \tuple{M, R}$, i.e., a structure where each propositional variable $p \in \tau$ is interpreted as a mapping $p_{\model{M}}\colon   M\longrightarrow [0,1]$, we can define a notion of truth-value at a world $w \in M$:

\begin{itemize}
\item$\|p[w]\|_{\model{M}}=p_{\model{M}}(w)$, for each $p\in \tau$;

\item $\|\circ( \f_0, \ldots, \f_n)[w]\|_{\model{M}}=$
\item[]$\circ(\|\f_0[w]\|_{\model{M}}, \ldots, \|\f_n[w]\|_{\model{M}})$, for $n$-ary connective $\circ$;

\item $\semvalue{ \Box \f [w]}_{\model{M}}=\inf\{\|\f [v]\|_{\model{M}}\mid v\in M, \tuple{w, v} \in R\}$;
\item $\semvalue{ \Diamond \f [w]}_{\model{M}}=\sup\{\|\f [v]\|_{\model{M}}\mid v\in M, \tuple{w, v} \in R\}$.
\end{itemize}

Every formula $\f$ can be said to be \emph{interpreted} in the model $\model{M}$ by the mapping $f_\f\colon M \longrightarrow [0, 1]$ defined as $w \mapsto \semvalue{\f[w]}_\model{M}$ (we also say that $\f$ \emph{defines} the mapping $f_\f$ in the model $\model{M}$). Given an MD-sentence $\tuple{\f_1, \ldots, \f_k; S}$, we write  $$\model{M} \models \tuple{\f_1, \ldots, \f_k; S}$$ if the formulas $\f_1, \ldots, \f_k$ respectively define mappings $f_1, \ldots, f_k$ in the model $\model{M}$ and $\tuple{f_1, \ldots, f_k} \in S$. 

As with the first-order case, from  the axioms and inference rules from~\citep{fagin} we need to modify only the following:

\paragraph{Axioms} 
\begin{itemize}
\item[(1)] $\tuple{\f_1, \ldots, \f_k, [0,1]^{M} \times \ldots \times  [0,1]^{M}}$ for any formulas $\f_1, \ldots, \f_k$.
\end{itemize}
\paragraph{Inference rules} 
%
%

\begin{itemize}
\item[(3)] From $$\tuple{\f_1, \ldots, \f_k; S}$$ infer $$\langle \f_1, \ldots, \f_k, \f_{k+1}, \ldots, \f_m; S \times [0,1]^{M}\times \ldots \times  [0,1]^{M}\rangle,$$
\end{itemize}


%
%
%
%
%

and we also need to modify the notion of good tuple for Rule~(7). Indeed, given an MD-sentence $\tuple{\f_1, \ldots, \f_k; S}$, now we say that a tuple $\tuple{f_1, \ldots, f_{k}} \in S$ is {\em good} if
\begin{itemize}
\item[(a)] $f_m = \circ (f_{m_1}, \ldots,   f_{m_j})$ whenever $\f_m = \circ(\f_{m_1}, \ldots, \f_{m_j})$,
\item[(b)]  $f_i(w)= \inf\{f_j( e) \mid e \in M, \tuple{w,e} \in R\}$ whenever $\f_i = \Box \f_j,$ for all $w \in M$,
\item[(c)]  $f_i(w)= \sup\{f_j( e) \mid e \in M, \tuple{w,e} \in R\}$ whenever $\f_i = \Diamond \f_j,$ for all $w \in M$.
\end{itemize}

As before, we get the following (since the interpretations of the propositional variables in $\tau$ is what determines a model over $\mathfrak{F}$): 
\begin{Lem}\label{min3} Let 
$\tuple{\f_1, \ldots, \f_k; S}$ be the premise of Rule~(7) and assume that $G=\{\f_1, \ldots, \f_k\}$ is closed under subformulas in the usual sense. Then, the conclusion $\tuple{\f_1, \ldots, \f_k; S'}$ is minimized.
\end{Lem}

Once more, closely following the argument from~\citep{fagin}, we may show that:

\begin{Thm}[Completeness of the logic of a fixed frame]
For $\Gamma$ a {\em finite} set of MD-sentences and $\gamma$ an MD-sentence, $\Gamma \vdash_{\mathfrak{F}} \gamma$ iff\/ $\Gamma \vDash_{\mathfrak{F}} \gamma$.
\end{Thm}

\begin{Rmk} \emph{ Observe that there is nothing canonical about the t-norm algebras on $[0,1]$: everything that has been done here could have been done for logics based on arbitrary fixed  lattices (see e.g.~\citep{haj}). }
\end{Rmk}

\section{Axiomatizations of prominent first-order (and modal) real-valued logics}\label{sec2}

Recall that, in the context of classical first-order logic, there are countably infinite models whose set of true first-order formulas is not recursively enumerable (for example the natural numbers with addition and multiplication), hence one cannot hope to find a recursive enumeration of the true formulas of any such model. By the L\"owenheim--Skolem theorem, the first-order sentences which are true in all countably infinite models coincide with the sentences that are true in all infinite models. On the other hand, the class of infinite models is axiomatizable in first-order logic: consider the theory formed by the sentences ``there are at least $n$ elements" for all natural numbers $n>0$. Hence, the first-order sentences which are true in all infinite models are recursively enumerable.  For if $\f$ is true in all countably infinite models, then $\neg \f$ cannot have any infinite model since otherwise $\neg \f$ would have a countably infinite model by the L\"owenheim--Skolem theorem.

On the other hand, let us analyze what happens in the real-valued case. In this section we will consider only the case of vocabularies \emph{without} equality. This is a very standard practice in mathematical fuzzy logic (e.g.~\citep{bel, scar, baaz, montagna, ta, hay}). Recall that neither \L ukasiewicz  nor Product first-order logic have a recursively enumerable set of validities with the semantics given on $[0,1]$ (see~\citep{scar} and~\citep{baaz}, respectively). On the other hand, G\"odel first-order logic is recursively axiomatizable~\citep{ta}, and   both \L ukasiewicz and Product logics can be axiomatized by the addition of an infinitary rule (see~\citep{hay, bel} and~\citep{montagna}, respectively).

\begin{Pro}\label{ls}
Let $\mathcal{L}$ be a first-order real-valued logic.\footnote{For example, $\mathcal{L}$ might be \L ukasiewicz, Product, or G\"odel first-order logic or, more generally, any first-order extension of an algebraizable logic in the sense of~\citep{de}.} Suppose that we have a countable vocabulary without equality. Then, for any $\mathcal{L}$-sentences $\f_1, \ldots, \f_k$ and any finite sequence $\tuple{r_1, \ldots, r_k}$ of reals from the interval $[0,1]$, there is an $\mathcal{L}$-model where $\f_1, \ldots, \f_k$ take values $r_1, \ldots, r_k$ respectively iff there is an $\mathcal{L}$-model with a countably infinite domain where $\f_1, \ldots, \f_k$ take values $r_1, \ldots, r_k$ respectively. 
\end{Pro}

\begin{proof} Suppose there is an $\mathcal{L}$-model, $\model{M}$, where $\f_1, \ldots, \f_k$ take values $r_1, \ldots, r_k$ respectively. By~\citep[Thm.~31]{de}, if $M$ is finite, one can build an  $\mathcal{L}$-model with a countably infinite domain where $\f_1, \ldots, \f_k$ take values $r_1, \ldots, r_k$ respectively (in fact there is a mapping between the two models that preserves the truth-values of all formulas). On the other hand, by~\citep[Thm.~30]{de}, if $M$ is infinite, one can build an  $\mathcal{L}$-model with a countably infinite domain where $\f_1, \ldots, \f_k$ take values $r_1, \ldots, r_k$ respectively (in such a way that the countable model can be chosen to be an elementary substructure of the original that preserves the truth-values of all formulas). 
\end{proof}

%

From this proposition and Theorem~\ref{t:Compl-MD-system-Countable-Domain} we immediately obtain that consequence from finite sets of premises in \L ukasiewicz, Product, and G\"odel first-order real-valued logic (without equality) is complete with respect to the MD-system of a countable domain:

\begin{Cor}\label{cont}
Let $M$ be a fixed countably infinite domain, let $\mathcal{L}$ be either \L ukasiewicz, Product, or G\"odel first-order real-valued logic without equality, and let $\vDash_\mathcal{L} $ be the corresponding consequence relation. For any finite set $\vectk{\f},\p$ of $\mathcal{L}$-sentences, we have:
\begin{center}$\tuple{\f_1; \{1\}}, \ldots, \tuple{\f_k; \{1\}}\vdash_M \tuple{\p; \{1\}}$ iff\/ $\f_1, \ldots, \f_k \vDash_\mathcal{L} \p$.
\end{center}
\end{Cor}

Observe that Corollary~\ref{cont} would fail in the presence of equality in the vocabulary. This is because general validity cannot be reduced to truth in any particular infinite (even if only countable) model. The reason is that, if $\psi$ is the first-order  sentence expressing that the size of the domain is 3 then $\neg \psi$ would hold in every infinite domain $M$, whereas this cannot be a  valid sentence in any of the logics we are considering here since $\psi$ holds in models with universes of size 3. Thus, we would have that $\not\vDash_\mathcal{L} \neg \p$ but $\vdash_M \neg \psi$.

The purpose of any completeness theorem is to obtain the equivalence between a universal statement (about validity) and an existential statement (about the existence of a proof). The claim of existence of a proof is a $\Sigma_1$ claim on the natural numbers when the proof system is arithmetizable. By Corollary~\ref{cont} and  since neither \L ukasiewicz  nor Product first-order logic has a recursively enumerable set of validities, our proof systems are not arithmetizable when the domain is infinite. 

\begin{Rmk} 
 \emph{Observe that, even in the case of classical logic (without equality --the situation with equality is analogous and dealt with in \S \ref{sec4}), the axiomatization we have presented here (when the domain in question is infinite) cannot be recursive due to Rule~(7), where most of the strength of the present approach resides (cf. Remark \ref{useful}). Naturally, there are much more fine-tuned axiomatizations of classical logic and many of the real-valued logics  under consideration here, but the sacrifice we have made in terms of the manageability of our proof system has been in the interest of generality, so we can encompass all these logics at once.
 }
\end{Rmk}

\begin{Rmk}
\emph{Readers not familiar with encoding syntax and proofs in  set theory  may skip this remark. By representing MD-sentences as sets and proofs as sequences of such sets (similarly as things are done in infinitary logic~\citep{di}), our notion of proof will be a $\Sigma_1$ predicate (in the L\'evy hierarchy) over the set of all sets hereditarily of some sufficiently large cardinality $\kappa$ (in fact cardinality $|2^\omega| +1$ would suffice for the case of a countably infinite fixed domain). Therefore, we have completeness in the same sense as it can be obtained in infinitary proof systems. Let us sketch the details of this formalization. Suppose that we fix a countable domain $M$. To each formula $\phi$ we can assign a G\"odel number $\ulcorner\phi\urcorner$   in the usual manner~\citep{kleene}.  We may then assign to each MD-sentence $\tuple{\phi_1, \ldots, \phi_k; S}$ the ``G\"odel set'' $\ulcorner\tuple{\phi_1, \ldots, \phi_k; S}\urcorner $  which is simply the set $\tuple{\ulcorner\phi_1\urcorner, \ldots, \ulcorner\phi_k\urcorner; S}$ (using the Kuratowski definition of ordered tuples). Take now the collection  $H(|2^\omega| +1)$ containing all sets $x$ hereditarily of cardinality $< |2^\omega|+1$ in the sense that $x$, its members, its members of members, etc., are all of cardinality $< |2^\omega|+1$. Consider now the following set-theoretic structure: $\tuple{H(|2^\omega|), \in \restriction H(|2^\omega|+1)}$. All G\"odel sets $\tuple{\ulcorner\phi_1\urcorner, \ldots, \ulcorner\phi_k\urcorner; S}$ are elements of $H(|2^\omega|+1)$. A collection $K \subseteq H(|2^\omega|+1)$ is said to be $\Sigma_1$  on $H(|2^\omega|+1)$ if it is definable in the structure $\tuple{H(|2^\omega|+1), \in \restriction H(|2^\omega|+1)}$ by a set theoretic formula equivalent to one built from atomic formulas and their negations by means of  the connectives $\wedge, \vee$, the restricted quantifier $\forall x \in y$ and the quantifier $\exists x$. One can check then that the notion of $\tuple{\ulcorner\phi_1\urcorner, \ldots, \ulcorner\phi_k\urcorner; S}$ being a provable formula in our system is $\Sigma_1$  on $H(|2^\omega|+1)$ because it claims the existence of a finite sequence of MD-sentences such that $\tuple{\ulcorner\phi_1\urcorner, \ldots, \ulcorner\phi_k\urcorner; S}$ is the last element of such sequence and every MD-sentence in it has been obtained by applying one of a finite number of rules to previous elements.}
\end{Rmk}

\begin{Rmk} \emph{From the results in~\citep{vid} we know that neither \L ukasiewicz  nor Product modal logics on the interval $[0,1]$ have recursively enumerable finitary ``global"  consequence relations.\footnote{This means that $\Gamma \vDash \phi$ if for all models based on frames from a given class, if $\Gamma$ is true at all points (or worlds) of the model, then $\phi$ is similarly true in all of them.} Hence, similarly to what we observed for the first-order case, the approach here does axiomatize the logics in question, but it gives recursive enumerability only when the frame is finite, not in general.}
\end{Rmk}

Part of the interest of the present approach is the uniformity it provides in axiomatizing the previously mentioned logics (which were known to be axiomatizable by other infinitary methods). We are essentially giving one recipe to deal with  all cases. Moreover, none of our rules are explicitly infinitary and the infinitary component of our formulas is hidden in the sets $S$.

Finally, in general, we are clearly axiomatizing more levels of formal reasoning than it could be done before, for preservation of value $1$ is a mere fraction of the possibilities that the present system actually handles. The system axiomatizes genuine real-valued reasoning in all of G\"odel,  \L ukasiewicz, and Product first-order (and modal) logics.

\section{A 0-1 law for MD-logics}\label{sec3}
In this section we want to establish a 0-1 law in the style of~\citep{fagin0} for certain MD-logics, namely those where we consider suitable finite subalgebras $\alg{A}$ of $[0,1]$ (of the form $\tuple{A,\wedge^\alg{A},\vee^\alg{A},\conj^\alg{A},\to^\alg{A},\0^\alg{A},\1^\alg{A}}$). For example, both G\"odel and \L ukasiewicz logic have multiple finitely-valued versions (though Product logic does not), and we will list some examples below. This restriction to the finite setting is because we wish to have, when our vocabularies are relational and finite, only a finite number of possible models on a given finite domain, in analogy to what happens in classical logic in~\citep{fagin0} (or in the many-valued case already considered in~\citep{bano}).

\begin{Exm}[The algebra of \L ukasiewicz 3-valued logic]
The algebra $$ \mbox{\L}_3=\tuple{\{0, \frac{1}{2}, 1\},\wedge^{ \mbox{\L}_3},\vee^{ \mbox{\L}_3},\conj^{ \mbox{\L}_3},\to^{ \mbox{\L}_3}, 0,1}$$ such that
\begin{flushleft}
\begin{itemize}
\item $\wedge^{ \mbox{\L}_3}(x, y) = \mbox{min}\{x,y\}$
\item  $\vee^{ \mbox{\L}_3}(x,y) = \mbox{max}\{x,y\}$

\item $\conj^{ \mbox{\L}_3}(x, y) = \mbox{max}\{0, x+y-1\}$

\item $\to^{ \mbox{\L}_3}(x, y) = \mbox{min}\{1, 1-x+y\}$

\end{itemize}
\end{flushleft}
More generally, we may consider any \L ukasiewicz $n$-valued logic by using the algebra $\mbox{\L}_n$ on the carrier set $\{0,\frac{1}{n-1}, \frac{2}{n-1}, \dots, \frac{n-2}{n-1}, 1\}$ and with the same definitions of operations.
\end{Exm}

\begin{Exm}[The algebra of G\"odel 4-valued logic]

The algebra $$ \mbox{G}_4=\tuple{\{0, \frac{1}{3},  \frac{2}{3}, 1\},\wedge^{\mbox{G}_4},\vee^{\mbox{G}_4},\conj^{\mbox{G}_4},\to^{ \mbox{G}_4}, 0,1}$$ such that
\begin{flushleft}
\begin{itemize}
\item $\wedge^{\mbox{G}_4}(x, y) = \conj^{ \mbox{G}_4}(x, y) = \mbox{min}\{x,y\}$
\item  $\vee^{ \mbox{G}_4}(x,y) = \mbox{max}\{x,y\}$

\item and for  $ \to^{\mbox{G}_4}$:
\begin{equation*}
  \to^{\mbox{G}_4}(x, y) = \begin{cases}
               1               &\mbox{if} \ x \leq \ y\\
               y              & \mbox{otherwise}.
           \end{cases}
\end{equation*}
\end{itemize}
\end{flushleft}

As in the previous example, we may also consider any G\"odel $n$-valued logic by using the algebra $\mbox{G}_n$ on the carrier set $\{0,\frac{1}{n-1}, \frac{2}{n-1}, \dots, \frac{n-2}{n-1}, 1\}$ and with the same definitions of operations.
\end{Exm}

Let us now recall some facts from classical finite model theory. Consider a purely relational vocabulary. A sentence is said to be \emph{parametric} in the sense of  Oberschelp in~\citep[p.~277]{ober} if it is a conjunction of sentences of the form $$\forall x_1, \ldots, x_k (\neq(x_1, \ldots, x_k) \rightarrow \phi(x_1, \ldots, x_k )),$$ where $\neq(x_1, \ldots, x_k)$ is the conjunction of negated equalities expressing that $x_1, \ldots, x_k$ are pairwise distinct, and $\phi(x_1, \ldots, x_k )$ is a quantifier-free formula where in all of its atomic subformulas $Rx_{i_1} \ldots x_{i_k}$ we have that $$\{x_{i_1}, \ldots, x_{i_k}\} = \{x_1, \ldots, x_k \}.$$
Moreover, for $k=1$, any formula $\forall x_1 \phi(x_1)$, where $\phi$ is a quantifier-free formula, is parametric. For example, $$\forall x \neg Rxx \land \forall x \forall y (x \neq y \to (Rxy \to Ryx))$$ is a parametric sentence, whereas $$\forall x\forall y\forall z (\neq(x,y,z) \to (Rxy \land Ryz \to Rxz))$$ is not.

Oberschelp's extension~\citep[Thm.~3]{ober} of Fagin's 0-1 law~\citep{fagin0} says:\emph{ on finite models and finite purely relational vocabularies,  for any class $K$ definable by a parametric sentence, any first-order sentence $\f$ will be almost surely true in members of $K$ or almost surely false}. By ``almost surely true" here we mean that the limit as $n$ goes to $\infty$ of the fraction of structures in $K$ with domain $\{1, \dots, n\}$  that satisfy a given sentence $\f$ is $1$ (and ``almost surely false" is defined analogously). Naturally, these fractions are well defined because there is only a finite number of possible structures on  finite vocabulary on the domain $\{1, \dots, n\}$. As we mentioned earlier, this fact is what motivates our restriction to finitely valued logics in this section.  A very accessible presentation of Oberschelp's result is~\citep[Thm.\ 4.2.3]{flum}.

An appropriate translation for our purposes from finitely-valued first-order logics into classical first-order logic is introduced in~\citep{bad}. Namely, for any sentence $\phi$ of a first-order logic based on a finite set $A \subseteq [0,1]$ of truth-values  and element $a \in A$, we have a first-order sentence $T^a(\phi)$ such that, for a certain theory $\Sigma$ (which can be written as a parametric sentence in the sense of Oberschelp~\citep{ober}), $T^a(\phi)$ is satisfied by a classical first-order model $\model{M}$ model of $\Sigma$ iff there is a corresponding first-order real-valued model $\model{M}^*$ where $\phi$ takes value exactly $a$.

The idea is that, starting with a relational vocabulary $\tau$ containing countably many predicate symbols $P^n_1, P^n_2, P^n_3, \ldots $ for each arity $n$, we can introduce a vocabulary $\tau^*$ containing predicate symbols $P_i^{na}$ for each $a \in A$ and each $n$ (the intuition here is that $P_i^{na}$ will hold of those objects for which $P_i^{n}$ takes truth-value $a$ in a given model), and the following translation from~\citep{bad} (where $\circ \in \{\vee, \wedge, \&, \rightarrow\}$):

\begin{center}{\small
\begin{tabular}{rcll}
$T^a(P^n_i x_1 \ldots x_n)$ & $=$ & 
$P_i^{na}x_1 \ldots x_n$  \,\,\,\, ($i \geq 1$)&
\\[0.5ex]
\\\vspace{0.4ex}
$T^a({\circ}(\vectn{\p}))$ & $=$ & $\bigvee\limits_{\substack {b_1,  \ldots, b_n \in A \\{\circ}^\alg{A}(b_1, \ldots, b_n) = a}}\bigwedge_{1\leq i\leq n} T^{b_i}(\p_i)$ & 
\\\vspace{0.4ex}
$T^a(\exists x\, \p)$ & $=$ & $\big(\bigvee\limits_{\substack {k \leq |A| \\b_1  \ldots b_k \in A \\ \max \{b_1, \ldots, b_k\} = a}} \bigwedge\limits_{i=1}\limits^{k} \exists x\,  T^{b_i}(\p) \big)\land $
\\\vspace{0.4ex}
 & & $ \hspace{1.4cm}  \wedge \forall y\, ( \bigvee\limits_{\substack {b \in A \\{b \leq a}}}     T^{b}(\p(y/x)))$
\\\vspace{0.4ex}
$T^a(\forall x\, \p)$ & $=$ & $\big(\bigvee\limits_{\substack {k \leq |A| \\b_1,  \ldots, b_k \in A \\ \min \{b_1, \ldots, b_k\} = a}} \bigwedge\limits_{i=1}\limits^{k} \exists x\,  T^{b_i}(\p) \big) \land$
\\\vspace{0.4ex}
 & & $ \hspace{1.4cm} \wedge \forall y\, ( \bigvee\limits_{\substack {b \in A \\{a \leq b}}}     T^{b}(\p(y/x))). $
\end{tabular}}
\end{center}

Observe how the translations of quantified formulas exactly describe the semantics of quantifiers in these finitely-valued logics (i.e.\ existential as maximum of the truth-values of instances of the formula and, dually, universal as minimum). We use classical disjunctions to run over all the possible choices of values $b_1,  \ldots, b_k \in A$ that would give value $a$ as their maximum (resp.\ minimum) and then write the conjunction of the necessary conditions that make sure that these $b_i$'s are indeed values of instances of $\p$ and any other instance would give a value smaller (resp.\ bigger) than $a$.

Next, we define the theory $\Sigma$ given by: 
 \[
 \forall x_1, \ldots, x_n (\bigvee_{\substack{a \in A }}P_i^{na}x_1 \ldots x_n), 
  \]
 \[   
  \forall x_1, \ldots, x_n (\neg (P_i^{na}x_1 \ldots x_n\land P_i^{nb}x_1 \ldots x_n)), \,\,\,\,\,\, \,\,\,\,\,\,
    \]
\[
\text{for } a, b \in A, a \neq b, P_i^{n} \in \tau.
\]
 
 For any $A$-valued model $\model{M}$ for the vocabulary $\tau$, we can introduce a classical model $\model{M}^*$  for the vocabulary $\tau^*$ such that for any $a \in A$,  the value of $\phi$ in $\model{M}$
is $a$ iff $\model{M}^*\models T^a(\phi)$.  $\model{M}^*$  is built by taking the same domain, $M$, as in $\model{M}$ and letting the interpretation of $P_i^{na}$ be the set of all elements from $M^n$ such that the interpretation of  $P_i^{n}$ in $\model{M}$ assigns them value $a$.  Observe that $\model{M}^*$  is a model of the theory $\Sigma$. By a similar process, from any model $\model{N}$ of $\Sigma$, we can extract an $A$-valued model  $\model{M}$ such that $\model{N} = \model{M}^*$.

\begin{Pro}\label{ob}
An MD-sentence $\tuple{\phi_1, \ldots, \phi_n; S}$ is almost surely true on $A$-valued models with finite domains iff $\bigvee_{\tuple{a_1, \ldots, a_n}\in S} (T^{a_1}(\phi_1) \wedge \ldots \wedge T^{a_n}(\phi_n))$ is almost surely true on the finite models of $\Sigma$.
\end{Pro}
\begin{proof}
Suppose that $\tuple{\phi_1, \ldots, \phi_n; S}$ is almost surely true on $A$-valued models with finite domains. But every finite model of $\Sigma$ can be seen as an $\model{M}^*$ for some finite $A$-valued model $\model{M}$, and $\model{M}^* \models \bigvee_{\tuple{a_1, \ldots, a_n}\in S} (T^{a_1}(\phi_1) \wedge \ldots \wedge T^{a_n}(\phi_n))$ iff  $\model{M}\models \tuple{\phi_1, \ldots, \phi_n; S}$. Hence, $\bigvee_{\tuple{a_1, \ldots, a_n}\in S} (T^{a_1}(\phi_1) \wedge \ldots \wedge T^{a_n}(\phi_n))$ is almost surely true on the finite models of $\Sigma$. The other direction follows by similar reasoning.
\end{proof}

Rewriting the theory $\Sigma$ with some care, one can turn it into a parametric sentence when $\tau$ is finite. For example, suppose that $\tau$ contains only a binary predicate $R$. Then, $\Sigma$ would have the form (for $a, b \in A, a \neq b$):
 \[
 \forall x_1 \forall x_2 (\bigvee_{\substack{a \in A }}R^{a}x_1 x_2), 
  \]
 \[   
 \forall x_1 \forall  x_2 (\neg (R^{a}x_1 x_2 \land R^{b}x_1 x_2)). \,\,\,\,\,\, \,\,\,\,\,\,
    \]
This can be put into parametric form by considering instead (for $a, b \in A, a \neq b$):
 \[
 \forall x_1 (\bigvee_{\substack{a \in A }}R^{a}x_1 x_1), 
  \]
 \[
 \forall x_1 \forall x_2 (x_1\neq x_2 \rightarrow \bigvee_{\substack{a \in A }}R^{a}x_1 x_2), 
  \]
   \[   
 \forall x_1 ( \neg (R^{a}x_1 x_1 \land R^{b}x_1 x_1)), \,\,\,\,\,\, \,\,\,\,\,\,
    \]
 \[   
 \forall x_1 \forall x_2 (x_1\neq x_2 \rightarrow \neg (R^{a}x_1 x_2 \land R^{b}x_1 x_2)). \,\,\,\,\,\, \,\,\,\,\,\,
    \]
\begin{Thm}[0-1 law for MD-logics based on finite algebras]\label{0-1} For any finite relational vocabulary, any MD-logic based on a finite set of truth-values, and any MD-sentence  $\tuple{\phi_1, \ldots, \phi_n; S}$, we have that $\tuple{\phi_1, \ldots, \phi_n; S}$ is almost surely satisfied by all finite models or $\tuple{\phi_1, \ldots, \phi_n; S}$ is almost surely not satisfied in all finite models.
\end{Thm}

\begin{proof} This is immediate by applying Oberschelp's version in~\citep{ober} of the 0-1 law in~\citep{fagin0} and our previous observations.  By Proposition~\ref{ob}, an MD-sentence $\tuple{\phi_1, \ldots, \phi_n; S}$ is almost surely true iff $\bigvee_{\tuple{a_1, \ldots, a_n}\in S} T^{a_1}(\phi_1) \wedge \ldots \wedge T^{a_n}(\phi_n)$ is almost surely true on the  parametric class defined by $\Sigma$.
\end{proof}

\begin{Rmk} \emph{One might wonder what is the relationship of Theorem~\ref{0-1} with the central result from~\citep{bano}. Suppose  we have a 1-dimensional sentence $\tuple{\phi; S}$. Then, applying the 0-1 law from~\citep{bano}, the value  $a_\phi$ that $\phi$ takes almost surely is in $S$ only if $\tuple{\phi; S}$ is almost surely true. Furthermore, if $\tuple{\phi; S}$ is almost surely true, then $a_\phi$ is in $S$ because $a_\phi$ is the value that $\phi$ takes almost surely. Thus, in the 1-dimensional case, both 0-1 laws are equivalent, but only the 1-dimensional case, and not the 2-dimensional case, is covered in~\citep{bano}.  Hence, the question really is whether for a \emph{finitely-valued} logic we would have that each MD-sentence is equivalent to a 1-dimensional sentence. In~\citep{fagin} it is shown that there is a 2-dimensional MD-sentence not equivalent to any 1-dimensional MD-sentence in logics based on the full interval $[0,1]$. Does the same hold for finitely-valued logics?}
\end{Rmk}

\section{The logic  of all domains}\label{sec4}
In this section, we will be using the same notion of model as in Definition \ref{model} and we will allow the presence of equality in the vocabulary. Now, for any given domain $M$, let us denote by  $\mathcal{L}_\mathrm{MD}(M)$ the finitary part of $\vDash_{M}$, that is, the set of all pairs $\tuple{\Gamma, \theta}$ where $\Gamma$ is a finite set of MD-sentences, $\theta$ is an MD-sentence, and every model over $M$ of $\Gamma$ is a model of $\theta$. In this section, we intend to take the next natural step and axiomatize the finitary part of the MD-logic of {\em all domains}, i.e.\ the logic $\bigcap_{M \, \text{a domain}} \mathcal{L}_\mathrm{MD}(M)$.   Let us denote this consequence relation simply as $\vDash$.

What kinds of inferences can appear in $\bigcap_{M \, \text{a domain}} \mathcal{L}_\mathrm{MD}(M)$? Clearly,  only those not mentioning any of the domains $M$, since otherwise the inference could be rather specific to a particular $M$. For example, an MD sentence where a domain $M'\neq M$ is mentioned in the set $S$ does not make sense in models based on the domain $M$, or rather it is always false. Thus, we set the goal of axiomatizing all the valid inferences $\Gamma \vDash \theta$  where $\Gamma \cup \{\theta\}$ is a finite set of  MD-sentences  of the form $\tuple{\f_1, \ldots, \f_k; S}$ with each $\f_i$ being sentences in the usual sense of a first-order predicate language and, hence, $S$ is simply a set of suitable tuples of truth-values (thus without a mention of any domain). 

\begin{Exm} The MD-sentence $\tuple{\f_1, \f_2; S}$ where $S= \{\tuple{0.5, 0.7}\}$ and  $\f_1= \forall x \, Px$ and $\f_2= \forall x (Px  \vee Ux)$ is an example of the kind of MD-sentence described above, where $\f_1$ and $\f_2$ are sentences in the usual first-order sense of not having any free individual variables.
\end{Exm}

Focusing on logical entailments between this kind of  MD-sentences, we can restrict attention (without loss of generality) to the models based in the following countable list of domains (let us call these the \emph{legal domains}):
\begin{itemize}
\item[(i)] the infinite domain of natural numbers $\{1, 2, \ldots\}$,
\item[(ii)] for each natural number $n$, a domain $D_n$ of size $n$ (making sure that they are pairwise disjoint and also disjoint from $\{1, 2, \ldots\}$).

\end{itemize}
This is because we have the following:

\begin{Pro}\label{ls2}
Any MD-sentence $\tuple{\f_1, \ldots, \f_k; S}$ (where, for each  $1 \leq i\leq k $, $\f_i$ is a first-order sentence in the usual sense) with an infinite model has a countable model too.
\end{Pro}

\begin{proof} Take $\model{M} \models \tuple{\f_1, \ldots, \f_k; S}$, so $\semvalue{\f_i}_\model{M} = s_i$ (for $1 \leq i\leq k $) for some $\tuple{s_1, \ldots, s_k}\in S$. 
By Proposition~\ref{ls}, then if $\model{M}$ is infinite, there is a countable model $\model{M}'$ such that $\semvalue{\f_i}_\model{M'} = s_i$ ($1 \leq i\leq k $) for $\tuple{s_1, \ldots, s_k}$, and hence   $\model{M}' \models \tuple{\f_1, \ldots, \f_k; S}$, as desired. 
\end{proof}

Consequently, if we denote the finitary part of the consequence relation over legal domains  by $\vDash^{\text{legal}}$, using Proposition~\ref{ls2}, we can see that $\Gamma \vDash^{\text{legal}} \theta$ iff $\Gamma \vDash \theta$ (where $\Gamma \cup \{\theta\}$ is a finite set of  MD-sentences  of the form $\tuple{\f_1, \ldots, \f_k; S}$ with each $\f_i$ being sentences in the usual sense of a first-order predicate language). This means that we can focus on axiomatizing $\vDash^{\text{legal}}$ for the class of MD-sentences that we have described in Proposition \ref{ls2} (even though proofs may involve manipulating all kinds of MD-sentences, like those we will introduce in the next paragraph). Therefore,  in what follows, we will restrict ourselves to consider \emph{legal models}, i.e., those based on a legal domain.

The idea is to assume MD-sentences to have the form $\tuple{\f_1(\overline{x}_{\f_1}), \ldots, \f_k(\overline{x}_{\f_k}); S}$ where each  $\f_i$ is a first-order formula whose free variables are $\overline{x}_{\f_i}= x_{i_{1}}, \ldots, x_{i_{n_i}}$ (for some $n_i \geq 0$), and $S \subseteq [0,1]^{\bigcup_{M \ \text{is legal}} M^{n_1}} \times \ldots \times [0,1]^{\bigcup_{M \ \text{is legal}} M^{n_k}}$\!.

\begin{Exm} Take a vocabulary $\tau$ with one binary predicate $R$. Then, we can build the MD-sentence $\tuple{Rxy,  \forall x \forall y (Rxy \rightarrow Ryx); S}$ where $$S =\{\tuple{f, 0.5} \mid \text{$f$ is  a mapping from $\bigcup_{M \ \text{is legal}} M^2$ into the interval $[0, 1]$} \}.$$ We want this sentence to 
be satisfied in a legal model $\model{M}$ with domain $M$ if the truth-value of $\forall x \forall y (Rxy \rightarrow Ryx)$ is $0.5$ and, furthermore, the interpretation  of $R$ in the model $\model{M}$ is the restriction to $M$ of one of the functions $f$ described in the definition of $S$ (which in this case, happens trivially). 
\end{Exm}

As expected, we may then write  $$\model{M} \models \tuple{\f_1(\overline{x}_{\f_1}), \ldots, \f_k(\overline{x}_{\f_k}); S}$$ if the formulas $\f_1(\overline{x}_{\f_1}), \ldots, \f_k(\overline{x}_{\f_k})$ respectively define  functions $f_1, \ldots, f_k$ on the domain $M$ such that there are $\tuple{f_1', \ldots, f_k'} \in S$ for which $f_1, \ldots, f_k$ are the respective restrictions to the domain $M$.

We transform  Axiom (1) into (1)$^*$:
$$\langle\f_1(\overline{x}_{\f_1}), \ldots, \f_k(\overline{x}_{\f_k}), [0,1]^{\bigcup_{M \ \text{is legal}} M^{n_1}} \times \ldots \times [0,1]^{\bigcup_{M \ \text{is legal}} M^{n_k}}\rangle$$ for all formulas $\f_1(\overline{x}_{\f_1}), \ldots, \f_k(\overline{x}_{\f_k})$.

Rules ($2$), ($4$), ($5$), and ($6$) from the original system are modified analogously into ($2$)$^*$, ($4$)$^*$, ($5$)$^*$ and ($6$)$^*$. 
Rule~(3) needs to be  modified as:
\begin{itemize}
\item[(3)$^*$] From $$\tuple{\f_1(\overline{x}_{\f_1}), \ldots, \f_k(\overline{x}_{\f_k}); S}$$ infer $$\langle\f_1(\overline{x}_{\f_1}), \ldots, \f_k(\overline{x}_{\f_k}), \f_{k+1}(\overline{x}_{\f_{k+1}}), \ldots, \f_m(\overline{x}_{\f_m}); S \times 
$$
$$ [0,1]^{\bigcup_{M \, \text{is legal}} M^{n_{k+1}}} \times \ldots \times [0,1]^{\bigcup_{M \, \text{is legal}} M^{n_{m}}}\rangle.$$
\end{itemize}

Finally, Rule~(7) is modified into  Rule~(7)$^*$  by changing the notion of good tuple. Indeed, given an MD-sentence $\tuple{\f_1(\overline{x}_{\f_1}), \ldots, \f_k(\overline{x}_{\f_k}); S}$, we will say that a tuple $\tuple{f_1, \ldots, f_{k}} \in S$ is {\em good} if for some legal domain $M$
\begin{itemize}
\item[(a)] $f_m\restriction M = \circ ((f_{m_1}\restriction M), \ldots, (f_{m_j}\restriction M))$ whenever\\ $\f_m(\overline{x}_{\f_m}) = \circ(\f_{m_1}(\overline{x}_{\f_{m_1}}), \ldots, \f_{m_j}(\overline{x}_{\f_{m_j}})),$
\item[(b)] $(f_i\restriction M)(e_{1}, \ldots, e_{n_j})= \inf\{(f_j\restriction M)(e_{1}, \ldots, e_{n_j}, e) \mid e \in M\}$ whenever $\f_i(\overline{x}_{\f_i}) = \forall y\, \f_j(\overline{x}_{\f_j}),$ for all $e_{1}, \ldots, e_{n_j} \in M^{n_j}$,
\item[(c)] $(f_i\restriction M)(e_{1}, \ldots, e_{n_j})= \sup\{(f_j\restriction M)(e_{1}, \ldots, e_{n_j}, e) \mid e \in M\}$  whenever $\f_i(\overline{x}_{\f_i}) = \exists y\, \f_j(\overline{x}_{\f_j})$, for all $e_{1}, \ldots, e_{n_j} \in M^{n_j}$.
\end{itemize} 

Rule~(7)$^*$ is clearly sound with respect to the relation $\vDash^{\text{legal}}$ since we are only considering models based on legal domains.\footnote{Notice that if in Rule~(7)$^*$ we had written ``for each legal domain'' instead of ``for some legal domain'' in the definition of a good pair, the soundness argument would not work for the resulting rule.}   Given this system, we denote the corresponding provability relation simply as $\vdash$.

\begin{Rmk}
\emph{ 
 Observe that the complexity of identifying an application of Rule~(7)$^*$   by constructing $S'$ is the same, generally speaking, as in the case of a fixed countably infinite domain and Rule~(7). This is because, for example,  in the latter case, in order to identify which tuples are in $S'$, one might still need to compute the infimum of an infinite set without any nice structure in general in the process of verifying the value of a universal quantification.
}
 \end{Rmk}

\begin{Lem}\label{minall} Let 
$\tuple{\f_1(\overline{x}_{\f_1}), \ldots, \f_k(\overline{x}_{\f_k}); S}$ be the premise of Rule~(7)$^*$ and assume that $G=\{\f_1(\overline{x}_{\f_1}), \ldots, \f_k(\overline{x}_{\f_k})\}$ is closed under subformulas in the usual sense. Then, the conclusion $\tuple{\f_1(\overline{x}_{\f_1}), \ldots, \f_k(\overline{x}_{\f_k}); S' }$ is minimized, that is, if  $\tuple{f_1, \ldots, f_k}\in S' $, then there is a legal model of  $\tuple{\f_1(\overline{x}_{\f_1}), \ldots, \f_k(\overline{x}_{\f_k}); S' }$, $\model{M}$,  such that for $1\leq i \leq k$ the interpretation of $\f_i(\overline{x}_{\f_i})$ is $f_i\restriction M$.
\end{Lem}

\begin{proof} Assume that $\tuple{f_1, \ldots, f_k}\in S'$. Since $G$ is closed under subformulas, there is a legal domain $M $  and  a subsequence of 
$\tuple{g_1, \ldots, g_j}$ of 
$\tuple{f_1, \ldots, f_k}$ such that $\tuple{g_1 \restriction M, \ldots, g_j \restriction M}$ 
determines interpretations on $M$ for the atomic formulas appearing in $G$, i.e., interpretations for the predicates of the vocabulary  $\tau$ in question. But this subsequence then defines a legal model $\model{M}$ based on the domain $M$ where the interpretations of $\f_1(\overline{x}_{\f_1}), \ldots, \f_k(\overline{x}_{\f_k})$ are as indicated by $\tuple{g_1 \restriction M, \ldots, g_j \restriction M}$. \end{proof}


\begin{Lem}\label{equiv2}The conclusion and premises of rules (2)$^*$, (3)$^*$, (4)$^*$, and (7)$^*$ are logically equivalent.

\end{Lem}

\begin{Lem}\label{min4}
Minimization is preserved by the rules (2)$^*$ and (4)$^*$, i.e.\ if the premises of the rules are minimized, then their conclusions are too.

\end{Lem}


With these key facts at hand, the soundness and completeness proof goes through basically as before:
\begin{Thm}[Completeness of the logic of all legal domains] Let $\Gamma \cup \{\theta\}$ be a finite set of MD-sentences  in  a first-order predicate language with equality. Then,
$\Gamma \vdash \theta$ iff\/ $\Gamma\vDash^{\text{legal}} \theta$.
\end{Thm}

\begin{Cor}[Completeness of the logic of all domains] Let $\Gamma \cup \{\theta\}$ be a finite set of MD-sentences of the form $\tuple{\f_1, \ldots, \f_k; S}$ with each $\f_i$ being a sentence in the usual sense of a first-order predicate language with equality. Then,
$\Gamma \vdash \theta$ iff\/ $\Gamma\vDash \theta$.
\end{Cor}

\begin{Rmk}
\emph{The approach provided in this section allows us now to axiomatize, in particular, the valid finitary \emph{consecutions} (i.e.\ pairs of the form $\tuple{\Theta, \theta }$ where  $\Theta$ is a finite set of first-order sentences and $\theta$ a first-order sentence such that the former logically entails the latter, see e.g.~\citep{cinnog}) of each of \L ukasiewicz, Product, G\"odel, and real-valued logics with equality. This is analogous to what we did in Corollary~\ref{cont}. Hence, to deal with the presence of equality in the logic, we had to leave the realm of the \emph{fixed} countable domain from Corollary~\ref{cont} and, instead, study all domains that can be distinguished by the expressive power of a first-order language with equality (namely, all finite domains in addition to a countably infinite ones).}
\end{Rmk}

Another interesting consequence of our approach is that we can provide a finitary axiomatization of the valid inferences on finite models for any real-valued logic. 
Let the class of \emph{legal$^*$ domains} be that of the legal domains minus the one countably infinite domain (so we are keeping only the finite domains). One can then modify the axiomatization given above by replacing the legal domains by the legal$^*$ ones. Clearly,  $\Gamma \vDash^{\text{legal}^*} \theta$ iff $\Gamma \vDash^{\text{finite}} \theta$, where  $\vDash^{\text{finite}}$ is the obvious logical consequence over all finite domains (notice that the legal domains are just a specific subset of all finite domains). Exactly as we did previously, we can obtain: 

\begin{Thm}[Completeness of the logic of all finite domains] Let $\Gamma \cup \{\theta\}$ be a finite set of MD-sentences  in  a first-order predicate language with equality. Then,
$\Gamma \vdash \theta$ iff\/ $\Gamma\vDash^{\text{legal}^*} \theta$ iff $\Gamma \vDash^{\text{finite}} \theta$.
\end{Thm}

By a well-known theorem of Trakhtenbrot \citep{tra}, the validities of  classical first-order logic on finite models are not recursively enumerable. In the real-valued setting, the result was generalized in \citep{bi} to a large class of logics. This entails that, once more, our  axiomatization cannot possibly be recursive. In fact, we can observe that the problem of determining whether $S' = \emptyset$ in Rule (7)$^*$ of our axiomatization is not recursively enumerable, which explains why our system is not recursive. This is because we can reduce the problem of whether a sentence of classical first-order logic is  valid in the finite to whether  $S' = \emptyset$. Take a first-order sentence
 $\f$ and let
  $\f_1(\overline{x}_{\f_1}), \ldots, \f_k(\overline{x}_{\f_k}), \f$ be the list of all its  subformulas. Consider now the MD-sentence $\tuple{\f_1(\overline{x}_{\f_1}), \ldots, \f_k(\overline{x}_{\f_k}), \f; S}$ (call it $\p$) where $S:= \{0,1\}^{\bigcup_{M \ \text{is legal}} M^{n_1}}\times \ldots \times  \{0,1\}^{\bigcup_{M \ \text{is legal}} M^{n_k}}\times \{0\}$.  Take now the MD-sentence obtained by applying our Rule~(7) to this sentence, $\tuple{\f_1(\overline{x}_{\f_1}), \ldots, \f_k(\overline{x}_{\f_k}), \f; S'}$ (call it $\psi'$). Observe that $\p$ and $\p'$ are  equivalent. Furthermore, $ \f$ is valid on all finite models  iff $\neg \f$ has no finite model iff $\p$ is not satisfiable in a finite domain iff $\p'$ is not satisfiable in a finite domain. Finally, by minimization and the semantics of MD-sentences, $\p'$ is not satisfiable in a finite domain iff $S'= \emptyset$.

\section{Conclusion}

In this article, we have proposed a new paradigm for dealing with inference in first-order (and modal) real-valued logics. By means of the syntax of multi-dimensional sentences, we have obtained a high level of expressivity that goes beyond the usual preservation of full truth given by the value $1$ and surpasses even the expressivity of rational Pavelka logic or other fuzzy logics with truth-constants (see e.g.~\citep{Esteva-Godo-Noguera:ExpandingWNM,Esteva-Godo-Noguera:Rational}). As usual, there is a trade-off between expressivity and effectivity of any logical formalism. In our case, we have presented axiomatic systems that are not finitistic in the sense of metamathematics~\citep{kleene} because MD-sentences contain a hidden infinitary component (that is, the sets $S$), but yet these systems involve only {\em finitary} rules. We have proved corresponding completeness theorems in a similar sense as they had been obtained with ad hoc {\em infinitary} proof systems for some particular real-valued logics (see~\citep{hay,montagna}), but now in a general, uniform, parameterized way. However, it should be stressed that on finite domains our proof systems become finitistic and everything works as in the
propositional case. Finally, sentences incorporating weights can be handled completely analogous to the way it is done in~\citep{fagin}.

\section*{Acknowledgements}

We are grateful to Xavier Caicedo for fruitful discussions that revealed  the usefulness of Oberschelp's result in the present context. We are also grateful to Moshe Vardi, Phokion Kolaitis, and Llu\'is Godo for useful remarks. Guillermo Badia was  supported by the Australian Research Council grant DE220100544. Finally, Carles Noguera and Guillermo Badia acknowledge support by the European Union's Marie Sklodowska--Curie grant no.\ 101007627 (MOSAIC project).

\end{document}